\documentclass[paperA4,12pt]{article}
\usepackage[cp1250]{inputenc}
\usepackage[T1]{fontenc}
\usepackage{amsfonts}
\usepackage{enumerate}
\usepackage{color}
\usepackage{amsmath}
\usepackage{amssymb}
\usepackage{amsthm} 
\usepackage{textcomp}
\usepackage{wrapfig}
\usepackage{yfonts} 
\usepackage{url}
\usepackage{fancyhdr}
\usepackage{geometry}
\usepackage{stmaryrd}
\usepackage{bbding}
\usepackage{enumitem} 
\usepackage{bbm}
\everymath{\displaystyle}
\theoremstyle{plain} 
\newtheorem{tw}{Theorem}

\newtheorem{lemma}{Lemma}
\newtheorem{corollary}{Corollary}

\theoremstyle{definition} 
\newtheorem{definition}{Definition}
\newtheorem{example}{Example}
\newtheorem{counterexample}{Counterexample}

\newcommand\mR{{\mathbb R}}
\newcommand\mN{{\mathbb N}}

\newcommand\cA{{\mathcal A}}

\newcommand\fS{{\textfrak{S}}}

\newcommand\bI{\mathbf{I}}

\newcommand{\rS}{\mathrm{S}}
\newcommand{\rL}{\mathrm{L}}
\newcommand{\rM}{\mathrm{M}}

\newcommand{\sP}{\mathsf{P}}
\newcommand{\cP}{\mathcal{P}}

\newcommand{\sX}{\mathsf{X}}
\newcommand{\sY}{\mathsf{Y}}

\renewcommand{\c}{\circ}

\newcommand\md{{\,\mathrm{d}}}

\newcommand{\sint}{(S)\int}
\newcommand{\nint}{(N)\int}
\def\Xint#1{\mathchoice
 {\XXint\displaystyle\textstyle{#1}}%
 {\XXint\textstyle\scriptstyle{#1}}%
 {\XXint\scriptstyle\scriptscriptstyle{#1}}%
 {\XXint\scriptscriptstyle\scriptscriptstyle{#1}}%
 \!\int}
 \def\XXint#1#2#3{{\setbox0=\hbox{$#1{#2#3}{\int}$}
 \vcenter{\hbox{$#2#3$}}\kern-.5\wd0}}
 
 \def\dashint{\Xint-}

\newcommand\calka[3][\circ]{\int_{#1,#2} #3 \md\mu } 
\newcommand\dcalka[3][\oplus]{\dashint_{#1,#2} #3 \md\mu } 
\newcommand\dolsug[2]{(S)\,\dashint_{#1}#2 \md\mu }

\renewcommand\ge{\geqslant}
\renewcommand\le{\leqslant}

\newcommand{\mI}[1]{\mathbbm{1}_{#1}}

\newcommand\ca{{\,\circ_1\,}}
\newcommand\cb{{\,\circ_2\,}}
\newcommand\cc{{\,\circ_3\,}}
\newcommand\ch{{\,\circ_h\,}}
\newcommand\ci{{\,\circ_i\,}}

\newcommand\lo{{\,\lozenge\,}}
\newcommand\trd{\,\triangledown\,}
\newcommand\trdd{\triangledown}
\newcommand\trr{\,\triangle\,}
\newcommand\caa{{\circ_1}}
\newcommand\cbb{{\circ_2}}
\newcommand\ccc{{\circ_3}}
\newcommand\cii{{\circ_i}}
\newcommand\loo{{\lozenge}}

\newcommand\st{\star}

\renewcommand\ge{\geqslant}
\renewcommand\le{\leqslant}

\title{On the Minkowski-H\"{o}lder type inequalities for generalized Sugeno integrals with an application}
\author{Michał Boczek\footnote{Corresponding author. E-mail adress: 800401@edu.p.lodz.pl; tel.: +48 42 6313859; fax.: +48 42 6363114. 
}, Marek Kaluszka
\\ 
{\emph{
\small{Institute of Mathematics, Lodz University of Technology, 90-924 Lodz, Poland}}}}
\date{}
\begin{document}
\maketitle

\begin{abstract}
In this paper, we use a~new method to obtain the necessary and sufficient condition guaranteeing the validity of the Minkowski-H\"{o}lder type inequality 
for the generalized upper Sugeno integral in the case of functions belonging to a~wider class than the comonotone functions. As a~by-product, we show that the Minkowski type inequality for seminormed fuzzy integral presented by Daraby $\cite{daraby4}$ is not true. Next, we study 
the Minkowski-H\"{o}lder inequality for the lower Sugeno integral and 
the class of $\mu$-subadditive functions introduced in \cite{boczek1}.
The results are applied to derive new metrics on the space of measurable functions in the setting of nonadditive
measure theory.   
We also give a~partial answer to the open problem $2.22$ posed in  $\cite{hutnik2}.$
\end{abstract}

{\it Keywords: }{Seminormed fuzzy integral; Semicopula; Monotone measure; Minkowski's inequality; H\"older's inequality; convergence in mean.}

\section{Introduction} The concepts of fuzzy measures and the Sugeno integral were introduced by Sugeno in $\cite{sugeno1}$ as a~tool for modeling nondeterministic problems.
The study of inequalities for the Sugeno integral was initiated by Rom\'an-Flores et al. $\cite{flores6}.$ Since then, the fuzzy integral counterparts of several classical inequalities, including Chebyshev's, Minkowski's and H\"older's inequalities have been given by Agahi et al. $\cite{agahi4},$ 
Klement et al.  $\cite{klement3}$,  Ouyang et al. $\cite{ouyang1, ouyang6}$, Wu et al. $\cite{lwu}$
and many other researchers. 
Most of them deal with  comonotone functions which highly limit the range of potential applications in probability, statistics, decision theory,  risk theory and others. 

Since many classical inequalities are free of the comonotonicity assumption, Agahi and Mesiar $\cite{agahi21}$  asked whether one could omit it. 
They gave a~version of the Cauchy–Schwarz inequality
without the comonotonicity condition for two classes
of Choquet-like integrals.  
In $\cite{boczek1}$  the Chebyshev type inequalities were provided for positively dependent functions which form a~wider class than the comonotone functions.  The aim of this paper is to present another 
inequalities for nonadditive integrals without the comonotonicity condtion.

The paper is organized as follows. In Section 2, we introduce a~new concept, called $\star$-associativity, which extends the notion of  comonotonicity. Next,  we obtain the necessary and sufficient conditions ensuring that the Minkowski-H\"older type inequality  holds for the generalized upper Sugeno integral and $\star$-associative functions. 
We give a~counterexample showing  that the Minkowski type inequality for seminormed fuzzy integral presented in $\cite{daraby4}$, Theorem $3.1,$ is false. The sufficient conditions for subadditivity of some 
functionals based on   the upper Sugeno integral are also provided. 
Section $3$ presents the Minkowski-H\"older type inequality for the generalized lower Sugeno integral and $\mu$-subadditive functions. The necessary and sufficient condition for subadditivity of the Sugeno integral
with respect to a subadditive measure is given. Finally, in Section $4$ we propose  new metrics on the space of
measurable functions when the involved measure is  monotone. We also give a~partial answer to the open problem posed by Borzová-Molnárová et al. $\cite{hutnik2}.$

\section{Inequalities for generalized upper Sugeno integral}
First, we introduce some basic definitions and properties.
Let $(X,\cA)$ be a~measurable space, where $\cA$ is a~$\sigma$-algebra of subsets of a~nonempty set $X.$ 
A~{\it monotone measure} on $\cA$  is a~nondecreasing set function 
$\mu\colon \cA\to [0,\infty]$ with $\mu(\emptyset)=0.$
We say that $\mu$ is {\it finite} if $\mu(X)<\infty.$ 
A~monotone measure $\mu$ is
{\it continuous from below} if $\lim_{n\to\infty}\mu(A_n)=\mu\big(\lim_{n\to\infty} A_n\big)$ for all  
$A_n\in\cA$  such that $A_n\subset A_{n+1},$ $n\in\mN.$

Let $Y=[0,m)$ or $Y=[0,m],$ where $0<m\le \infty$; usually 
$Y=[0,1]$, $Y=[0,\infty)$ or $Y=[0,\infty].$    
The operator  $\circ\colon Y^2\to Y$ is said to 
be {\it nondecreasing} if
$a\circ b\le x\circ y$ for $a\le x$, $b\le y$.
We say that $\circ\colon Y^2\to Y$ is {\it right-continuous} if $\lim_{n\to\infty}(a_n\circ b_n)=a\circ b$ for all $a_n,b_n,a,b\in Y$ such that $b_n \searrow b$ and $a_n\searrow a.$ Hereafter,  $c_n\searrow c$ means that $\lim_{n\to\infty} c_n=c$ and $c_n>c_{n+1}$ for all $n.$  

Recall that  $f,g\colon X\to Y$ are {\it comonotone} on $D$ if $\big(f(x)-f(y)\big)\big(g(x)-g(y)\big)\ge 0$ for all $x,y\in D.$ 
If $f$ and $g$ are comonotone on $D,$  
then for any  $t\in Y$ either $(D\cap \lbrace f\ge t\rbrace) \subset (D\cap \lbrace g\ge t\rbrace)$ or $(D\cap\lbrace g\ge t\rbrace )
\subset (D\cap \lbrace f\ge t\rbrace),$
where $\lbrace f\ge t\rbrace=\lbrace x\in X\colon f(x)\ge t\rbrace.$ 

Now we will generalize the concept of comonotonicity. 

\begin{definition} Given an~operator $\star\colon Y^2\to Y,$
we say that $f,g\colon X\to Y$ are {\it $\star$-associated} on $D$ 
if for any nonempty and measurable subset  $A\subset D$, 
\begin{align}\label{funkcjeskojarzone}
 \inf_{x\in A} \left\{ f(x)\star g(x) \right\}=\inf_{x\in A} f(x)\star \inf_{x\in A} g(x).
 \end{align} 
\end{definition}

\noindent From now on,  $a\wedge b=\min\{a,b\},$  
$a\vee b=\max\{a,b\}$ and $a_+=a\vee 0.$ 
\medskip

\begin{example} \rm 
Any functions $f,g\colon X\to Y$ are  $\wedge$-associated on $X$. 
\end{example}

\begin{example} \rm 
Any comonotone functions $f,g\colon X\to Y$ are  $\star$-associated on $X$ if the operator $\star$ is nondecreasing and right-continuous. Indeed,  $\inf_{x\in A} \lbrace f(x)\star g(x)\rbrace \ge s\star t$ for all $A\subset X,$ where  $s=\inf_{x\in A} f(x)$  and $t=\inf_{x\in A} g(x).$ Let $\varepsilon>0,$ $A\subset X$ and $B=\lbrace x\in A\colon f(x)<s+\varepsilon\rbrace$ and $C= \lbrace x\in A\colon g(x)<t+\varepsilon\rbrace.$  From the comonotonicity we obtain that $B\cap C\neq \emptyset$ as $B\subset C$ or $C\subset B$. Thus $\inf_{x\in A}\lbrace f(x)\star g(x)\rbrace\le (s+\varepsilon )\star (t+\varepsilon ).$ Because of  the right-continuity of $\star,$ we get the assertion.
\end{example}

\begin{example}\rm Let $f,g\colon X\to Y$ be measurable functions and $g=b\mI{B}$ for $b\in Y$, where $\mI{B}$ denotes  the indicator of  $B\subset X,$  $B\cap \{f\ge b\}\neq B$ and 
$B\cap \{f\ge b\}\neq \{f\ge b\}.$ 
Let $\star\colon Y^2\to Y$ be a nondecreasing and right-continuous operator.
If $x\star 0=0$ for all $x\in X$, then $f,g$ are $\star$-associated on $X,$ but not comonotone.
Indeed, if $A\backslash B= \emptyset,$  then 
$A\subset B$ and  
\begin{align*}
\inf_{x\in A}\left\{f(x)\star g(x)\right\}=\inf_{x\in A}\left\{f(x)\star b\right\}=\inf_{x\in A}f(x)\star \inf _{x\in A}g(x).
\end{align*}
If $A\backslash B\neq \emptyset,$ then 
\begin{align*}
\inf_{x\in A}\left\{f(x)\star g(x)\right\}&=\inf_{x\in A\cap B}\left\{f(x)\star b\right\}\wedge \inf_{x\in  A\backslash B}\left\{f(x)\star 0\right\}=0\\
&=\inf_{x\in  A}f(x)\star \inf_{x\in  A}g(x).
\end{align*}
\end{example}

\begin{example}\rm Suppose $\star$ is a nondecreasing operator such that $0\star y=y\star 0=0$ for all $y\in Y.$
Let $f=b\mI{B}+c\mI{C}$ and $g=b\mI{B}+c\mI{D}$, where $b,c\in Y,$  $0<b\wedge c,$ and $B,\,C$ are nonempty sets such that $B\cap C=\emptyset$ and 
$D=X\backslash (B\cup C)\neq \emptyset.$  
Clearly,  $f$ and $g$ are $\star$-associated on $X$, but not comonotone.
\end{example}

\begin{example} \label{com} \rm Functions $f,g$ are $+$-associated  
if and only if they are  comonotone. 
In fact, the condition $\eqref{funkcjeskojarzone}$
for $\star=+$ and $A=\{x,y\}$ is equivalent to 
$(a+b)\wedge 0=(a\wedge 0)+(b\wedge 0)$ with $a=f(x)-f(y)$ and $b=g(x)-g(y)$, and 
this implies that $ab\ge 0.$ 
\end{example}
\medskip

{\bf Open problem 1.} Does there exist an~operator $\c\neq +$ such that 
the $\c$-associativity property is equivalent to the comonotonicity property?
\medskip

Now we are ready to  present  the Minkowski-H\"older type inequality for the {\it generalized upper Sugeno integral} of the form 
\begin{align}\label{calka1}
\calka{D}{f}:=\sup_{t\in Y} \left\{t\circ \mu \big(D\cap \lbrace f\ge t\rbrace\big)\right\},
\end{align} 
where $f\colon X\to Y$ is a~measurable function,   $\mu$ is a~monotone measure on $\cA$ and $\circ\colon Y\times \mu(\cA)\to [0,\infty]$ is a~nondecreasing operator.   
The functional in $\eqref{calka1}$ is the universal integral in the sense of  Definition $2.5$ in $\cite{klement3}$ if $\circ$ is the pseudomultiplication function (see  $\cite{klement3}$, Definition $2.3$). 

Put $\mu (\cA)=\{\mu(A)\colon A\in \cA\}$ and $\mu (\cA\cap D)=\{\mu(A\cap D)\colon A\in \cA\}$.
The following theorem gives an~answer  to open problems from $\cite{agahi4},$ $\cite{boczek1}$
and $\cite{ouyang6}.$

\begin{tw}\label{ctw7}
Assume the operators  $\star, \loo\colon Y^2\to Y$ and $\circ_i\colon Y\times \mu(\cA)\to Y$ are such that $\loo $ and $\circ _i$ are nondecreasing and $0\ci x=y\ci 0=0$ for all $y\in Y,$ $x\in \mu (\cA)$ and  $i=1,2,3$. Suppose  $\phi_i\colon Y\to Y$ are increasing and $\phi _i\big(Y\big)=Y$ for all $i.$   
Suppose  also that $f$ and $g$ are $\star$-associated on $D\subset X$. Then the Minkowski-H\"older type inequality
\begin{align}\label{c51}
\phi^{-1}_1\Big(\calka[\caa]{D}{\phi_1(f\st g)}\Big)\le \phi^{-1}_2\Big(\calka[\cbb]{D}{\phi_2(f)}\Big)\lo\phi^{-1}_3\Big(\calka[\ccc]{D}{\phi_3(g)}\Big)
\end{align}
is satisfied if and only if for all $a,b\in Y$ and $c\in\mu (\cA \cap D)$
\begin{align}\label{c52}
\phi^{-1}_1 \big(\phi _1(a\star b)\ca c\big)\le \phi^{-1}_2\big(\phi _2(a)\cb c\big)\lo \phi^{-1}_3\big(\phi _3(b)\cc c\big).
\end{align} 
\end{tw}

\begin{proof} 
Arguing  as in the proof of  Lemma $3.8$ in $\cite{suarez},$  we can show that
\begin{align}\label{c53}
\calka[\cii]{D}{f}=\sup_{A\subset D,A\in\cA}\left\{\inf_{x\in A} f(x)\circ_i \mu(A)\right\}
\end{align}
for all $i$ (see also $\cite{hutnik1}$, Theorem $2.2$). To shorten the notation, we  
write $\sup _A$ instead of $\sup_{A\subset D,A\in\cA}.$
From the continuity of $\phi _1$ and $\eqref{c53}$ we get
\begin{align*}
L:=\phi^{-1}_1\Big(\calka[\caa]{D}{\phi_1(f\st g)}\Big)=\sup_{A}  \phi^{-1}_1 \Big( \phi _1\big(\inf_{x\in A}\left\{ f(x)\star g(x)\right\}\big) \ca\mu(A)\Big). 
\end{align*}
Since   $f$ and $g$ are $\star$-associated, we have 
\begin{align*}
L=\sup_{A}  \phi^{-1}_1\Big( \phi_1\big(\inf_{x\in A} f(x)\star\inf_{x\in A} g(x)\big) \ca\mu(A)\Big). 
\end{align*}
Combining $\eqref{c52}$ with 
the monotonicity of $\loo$ and $\phi ^{-1}_i$  yields
\begin{align*}
L&\le \sup_{A}\left\{\phi^{-1}_2\Big( \phi _2\big(\inf_{x\in A} f(x)\big)\cb\mu(A)\Big)\lo\phi^{-1}_3\Big( \phi _3\big(\inf_{x\in A} g(x)\big)\cc\mu(A)\Big)\right\}\\&\le \bigg(\sup_{A}\phi^{-1}_2\Big( \phi _2\big(\inf_{x\in A} f(x)\big)\cb\mu(A)\Big)\bigg)\lo \bigg(\sup_{A}\phi^{-1}_3\Big( \phi _3\big(\inf_{x\in A} g(x)\big)\cc \mu(A)\Big)\bigg)\\
&=\phi^{-1}_2\Big(\calka[\cbb]{D}{\phi_2(f)}\Big)\lo\phi^{-1}_3\Big(\calka[\ccc]{D}{\phi_3(g)}\Big).
\end{align*}
To obtain the necessary condition $\eqref{c52}$, put  $f=a\mI{A}$ and $g=b\mI{A}$ in  $\eqref{c51},$ where $c=\mu(A)\le \mu(D)$ and 
 $a,b\in Y.$
 \end{proof}

Observe that the assumption $0\ci x=y\ci 0=0$ is  used only in the proof of the necessity of  condition $\eqref{c52}.$ Moreover, 
the condition $\eqref{c52}$ is  sufficient for inequality $\eqref{c51}$ to hold if we set $Y=\mR$ in $\eqref{calka1}$ and  both $f$ and $g$ are bounded from below.

\begin{example}
Let $a\star b=a\lo b=a+b-ab,$ where $a,b\in Y=[0,1]$ and let $\circ_i=\cdot$ for all $i.$ Put $\phi_i(x)=x^{p_i}$  and $c_i=c^{1/p_i},$ where  $p_i>0$ for all $i.$ The condition $\eqref{c52}$ takes the form
\begin{align}\label{c54}
0\le a(c_2-c_1)+b(c_3-c_1)+ab(c_1-c_2c_3)
\end{align}
and holds if and only if  $p_1\le p_j$ for $j=2,3;$ in order to see this, put $a=1,$ $b=0$ as well as 
$a=0,$ $b=1$ in $\eqref{c54}$ and observe that
\begin{align*}
a(c_2-c_1)+b(c_3-c_1\big)+ab (c_1-c_2c_3)\ge ab\big(c_2-c_1+c_3(1-c_2)\big)\ge 0.
\end{align*}
\end{example}
      
We recall that the {\it Sugeno integral} and the {\it Shilkret integral} are given by
\begin{align}\label{sug}
\sint _D f \md\mu &:=\sup _{y\in Y}\left\{y\wedge \mu \big(D\cap \{f\ge y\}\big)\right\},\\ \label{shi}
\nint _D f \md\mu &:=\sup _{y\in Y}\left\{y\cdot \mu \big(D\cap\{f\ge y\}\big)\right\},
\end{align}
respectively, where $Y=[0,m]$ or  $Y=[0,m)$  with $0<m\le \infty,$  see $\cite{shilkret,sugeno1, zwang1}.$ 
\begin{corollary} 
Assume $\star\colon Y^2\to Y$ is nondecreasing, $f,g\colon X\to Y$ are $\star$-associated on $D$ and $\phi _i\colon Y\to Y$ are increasing functions such that $\phi _i\big(Y\big)=Y$
for $i=1,2,3.$  The following Minkowski--H\"{o}lder type inequality  
\begin{align*}
\phi^{-1}_1\Big(\sint_D\phi _1(f\star g)\md\mu\Big)\le \phi^{-1}_2\Big(\sint_D\phi _2 (f)\md\mu\Big)\star\phi^{-1}_3\Big(\sint_D\phi _3(g)\md\mu\Big)
\end{align*}
holds true if and only if for $a,b\in Y$ and $c\in\mu (\cA \cap D)$ 
\begin{align}\label{wu}
(a\star b)\wedge \phi^{-1}_1 (c)\le \big(a\wedge  \phi^{-1}_2 (c)\big)\star (b\wedge \phi^{-1}_3 (c)\big).
\end{align} 
\end{corollary}

The above result generalizes Theorem $3.1$ from $\cite{agahi4}$ and Theorem $3.1$ from $\cite{lwu}.$ In fact, 
since $a\vee b\le a\star b,$ we have  
$c\le a\vee c\le a\star c$, $c\le c\star b$ and 
$c\le c\star c,$
so 
\begin{align*}
(a\star b)\wedge c&\le (a\star b)\wedge  (a\star c)\wedge  (c\star b)
\wedge  (c\star c)=(a\wedge c)\star(b\wedge c).
\end{align*}
It follows from the assumption  $\phi _1\ge \phi _j$ for $j=2,3$ that
$$
(a\star b)\wedge \phi ^{-1}_1(c)
\le \big(a\wedge \phi ^{-1}_1(c)\big)\star\big(b\wedge \phi ^{-1}_1(c)\big) 
\le \big(a\wedge \phi ^{-1}_2(c)\big)\star\big(b\wedge \phi ^{-1}_3(c)\big).$$
Thus, the condition  $\eqref{wu}$ holds.
\medskip

Suppose that $\rS\colon [0,1]^2\to [0,1]$ is a~{\it semicopula} (also called 
a~$t$-{\it seminorm}), i.e., a~nondecreasing function 
with the neutral element equal to $1.$  It is clear that $\rS(x,y)\le x\wedge y$ 
and $\rS(x,0)=0=\rS(0,x)$ for all $x,y\in [0,1]$
(see $\cite{bas,dur,klement2}$).  We denote the class of all semicopulas
by $\fS$. 
There are three important examples of semicopulas: $\rM,$ $\Pi$ and $\rS_{\rL},$ where $\rM(a,b)=a\wedge b,$ $\Pi(a,b)=ab$ and $\rS_{\rL}(a,b)=(a+b-1)_+,$
usually called the {\it Łukasiewicz t-norm} $\cite{klement2}.$ 

Given $\rS\in \fS$, the {\it seminormed fuzzy integral}  is defined by
\begin{align}\label{semi}
\calka[\rS]{D}{f}:=\sup_{t\in [0,1]} \rS\big(t,\mu (D\cap\lbrace f\ge t\rbrace )\big),
 \end{align}  
see $\cite{ouyang1,suarez}.$ Replacing semicopula $\rS$ with $\rM$, we get the  Sugeno integral $\eqref{sug}$ for $Y=[0,1].$ 
Moreover, if  $\rS=\Pi,$ then  we get  the {\it Shilkret integral} $\eqref{shi}$ for $Y=[0,1]$. 
 
\begin{corollary} Let $\rS\in \fS$  and $f,g\colon X\to [0,1]$ be $\star$-associated, 
where $\star\colon [0,1]^2\to [0,1]$ is a~nondecreasing operator. 
Let $0<p<\infty$ and $\mu(\cA)=1.$ The following inequality holds   
\begin{align}\label{c55}
\Big(\calka[\rS]{D}{(f\st g)^p} \Big)^{1/p}\le \Big(\calka[\rS]{D}{f^p} \Big)^{1/p}\star\Big(\calka[\rS]{D}{ g^p}\Big)^{1/p}
\end{align}
if and only if 
\begin{align*}
\rS\big((a\star b)^p,c\big)^{1/p}\le \rS\big(a^p,c\big)^{1/p}\star \rS\big(b^p,c\big)^{1/p}
\end{align*} 
for $a,b\in [0,1]$ and $c\in\mu (\cA \cap D).$
\end{corollary}

Daraby and Ghadimi $\cite{daraby4}$ claim that the inequality 
$\eqref{c55}$ is satisfied if
\begin{align}\label{daraby}
\rS(a\star b,c)\le \big(\rS(a,c)\star b\big)\wedge \big(a\star \rS(b,c)\big),\quad a,b,c\in [0,1],
\end{align}
under the assumption of continuity of monotone measure $\mu$ (see  $\cite{daraby4}$, Theorem $3.1$). We present a~counterexample showing that this result is not true.
\medskip

\begin{counterexample}\label{kontrprzyklad}
Put $A=X=[0,1],$ $s=1,$ $\mathrm T=\rS_{\rL},$ $a\star b=(a+b)\wedge 1$ and $f(x)=g(x)=0.5\sqrt{x},$ $x\in [0,1],$ in Theorem $3.1$ from $\cite{daraby4}$.
Clearly, $f$ and $g$ are comonotone. Let $\mu$ be the Lebesgue measure.  Due to the property $a\star b=b\star a$,  the condition $\eqref{daraby}$ is satisfied  if and only if
\begin{align*}
\rS_{\rL}\big((a+b)\wedge 1,c\big)\le \big(\rS_{\rL}(a,c)+b\big)\wedge 1
\end{align*}
for all $a,b,c\in [0,1].$ Since $\rS_{\rL}\le 1,$ it suffices to show  that $\rS_{\rL}\big((a+b)\wedge 1,c\big)\le \rS_{\rL}(a,c)+b.$ In fact, if $a+b\le 1,$ then $\rS_{\rL}(a+b,c)\le \rS_{\rL}(a,c)+b$ (see also $\cite{klement2},$ Remark $5.13$\, $(iii)$). Otherwise, 
$$c\le (a+c-1)_++(1-a)_+=(a+c-1)_+-(a+b-1)+b\le \rS_{\rL}(a,c)+b.$$
Easy computations show that
\begin{align*}
\calka[\rS_{\rL}]{X}{f}&=\sup_{t\in [0,1]} \big(t+\mu\big(\lbrace f\ge t\rbrace \big)-1\big)_+=\sup_{t\in [0,1]} \big(t-4t^2\big)_+=0.0625,\\
\calka[\rS_{\rL}]{X}{(f\st g)}&=\sup_{t\in [0,1]} \big(t+\mu\big(\lbrace f\star g\ge t\rbrace \big)-1\big)_+=\sup_{t\in [0,1]} \big(t-t^2\big)_+=0.25.
\end{align*}
Hence, $0.25=\calka[\rS_{\rL}]{X}{(f\st g)}>\calka[\rS_{\rL}]{X}{f}\star \calka[\rS_{\rL}]{X}{g}=0.125.$
\end{counterexample}
\vspace{0.5cm}

Now we  focus on the subadditivity property of the generalized upper Sugeno integral $\eqref{calka1},$ that is, 
\begin{align}\label{suba}
\calka{X}{(f+g)}\le \calka{X}{f}+\calka{X}{g},
\end{align}
as this property is very important for applications. Let us recall that 
$+$-associativity is equivalent to comonotonicity, see Example  $\ref{com}.$

\begin{corollary}\label{cwn1} Let $Y=[0,m]$ or $Y=[0,m)$ for $0<m\le \infty$
and let $\circ \colon Y^2\to Y$ be 
a~nondecreasing operator such that $0\circ y=y\circ 0=0$ for all $y.$
The functional $\eqref{calka1}$ is subadditive for  comonotone functions 
$f,g\colon X\to Y$ such that $f+g\in Y$ if and only if 
$(a+b)\circ c\le (a\circ c)+(b\circ c)$ for  $a,b\in Y$, $a+b\in Y$ and $c\in\mu (\cA)=Y$.
\end{corollary}

 It follows from Corollary $\ref{cwn1}$ that both the Sugeno integral $\eqref{sug}$ and the Shilkret integral $\eqref{shi}$
are subadditive for comonotone functions while
the opposite-Sugeno integral $\int _{\rS_{\rL},D}f\md\mu$
$\cite{imaoka}$ is not.

\begin{corollary}
Let $\c=\rS\in \fS.$ Then the subadditivity property $\eqref{suba}$ is fulfilled for any monotone measure $\mu$ such that $\mu (X)\le 1$ and  comonotone functions $f,g\colon X\to [0,1]$ such that 
$f+g\in [0,1]$ 
if and only if
\begin{align}\label{c57}
\rS\big(a+b,c\big)\le \rS(a,c)+\rS(b,c)
\end{align}
for all $a,b,c\in [0,1],$    $a+b\in [0,1].$
\end{corollary}

Borzová-Molnárová et al. $\cite{hutnik1}$ showed that the inequality 
$\eqref{c57}$ is satisfied for each 
semicopula with concave horizontal sections $x\mapsto \rS(x,y)$. 
An~example is the Marshall–Olkin semicopula 
$\rS_{\alpha,\beta}(x,y)=(x^{1-\alpha}y)\wedge (xy^{1-\beta}),$ where $\alpha,\beta \in [0,1].$
Observe that if $f=\mI{A}$ and $g=\mI{B}$ for $A\cup B=X$ and $A\cap B=\emptyset$, then  the inequality  
$\eqref{suba}$ is of the form
$\mu (X)\le \mu(A)+\mu (B)$ for any semicopula  $\rS.$
Thus, the seminormed fuzzy integral 
is not subadditive if $\mu (A)+\mu (B)<\mu (X).$

We say that $\mu\colon\cA\to Y$ is {\it subadditive} if it is a monotone measure and $\mu(A\cup B)\le\mu(A)+\mu(B)$ for all $A,B\in\cA .$
The class of subaditive measures is quite wide and includes  the following monotone measures:  $\lambda$-measure of Sugeno for $\lambda\in \big(-1/\mu(X),0\big)$ (see  $\cite{zwang1}$, Definition $4.3$); the 
{\it plausibility measure} $\cite{zwang1}$; the {\it coherent measure} $\mu (A)=\sup _{\sP\in \cP}\sP(A)$, where $\cP$ is a~set of probability measures
$\cite{folmer}$;  the {\it possibility measure} $\mu (A)=\sup _{x\in A}\psi (x)$, where $\psi\colon X\to Y$ $\cite{zwang1},$ the {\it distortion measure}
$\mu (A)=g\big(\sP(A)\big)$, where $\sP$ is probability measure and  $g\colon [0,1]\to Y$ is such that $g(x+y)\le g(x)+g(y)$  $\cite{ruschendorf}$ and {\it uncertain measure} $\cite{bliu11},$ among others.

\begin{tw}\label{tw_subad} Suppose $Y=[0,m]$ or 
$Y=[0,m)$ with $0<m\le \infty$ and suppose  $\c\colon Y^2\to Y$  is a~nondecreasing operator such that  
$x\c (y+z)\le (x\c y)+(x\c z)$ for all $x,y,z\in Y$ such that $y+z\in Y.$ Suppose also that  $(ax)\c y\le a^q(x\c y)^r$ for some $q,r>0$ and for all $x,y,z\in Y, a>1$ such that  
$ax\in Y.$ 
Then for any $p>0$, any  subadditive measure $\mu$   and any measurable functions $f,g\colon X\to \mR$  such that 
 $|f+g|^p,$ $|f|^p,$ $|g|^p\in Y,$ we have
\begin{align}\label{c58}
\Big(\calka{X}{|f+g|^p}\Big)^{1/(pq+1)}&\le \Big(\calka{X}{|f|^p}\Big)^{r/(pq+1)}+
\Big(\calka{X}{|g|^p}\Big)^{r/(pq+1)}.
\end{align} 
\end{tw}

\begin{proof}  Without loss of generality, assume that 
$\calka{X}{|f|^p}+\calka{X}{|g|^p}<\infty.$ 
Clearly, $\lbrace |f+g|\ge t^{1/p}\rbrace\subset \lbrace |f|\ge \lambda t^{1/p}\rbrace \cup\lbrace |g|\ge (1-\lambda) t^{1/p}\rbrace$ for  $t\in Y$ and $\lambda\in (0,1).$ Thus, by the subadditivity of $\mu$ and monotonicity of $\c,$ we have
\begin{align*}
t\c \mu\big(\lbrace |f+g|^p\ge t\rbrace \big)&\le t\c \Big(\mu \big(\lbrace |f|^p\ge \lambda ^pt\rbrace\big)+\mu \big(\lbrace |g|^p\ge (1-\lambda)^p t\rbrace\big) \Big).
\end{align*}
From the assumptions on $\c,$  we get
\begin{align*}
\calka{X}{|f+g|^p}&\le 
\sup _{t\in Y}\left\{ t\c \mu \big(\lbrace |f|^p\ge \lambda ^pt\rbrace\big)\right\}+
\sup _{t\in Y}\left\{ t\c \mu \big(\lbrace |g|^p\ge (1-\lambda )^pt\rbrace\big)\right\}\nonumber\\
&\le 
\sup _{y\in \lambda ^pY}\left\{ \tfrac{y}{\lambda^p}\c \mu\big(\lbrace |f|^p\ge y\rbrace\big)\right\}+
\sup _{y\in (1-\lambda)^pY}\left\{\tfrac{y}{(1-\lambda)^p}\c \mu \big(\lbrace |g|^p\ge y\rbrace\big)\right\}\nonumber\\
&\le \lambda^{-pq}\,\Big(\calka{X}{|f|^p}\Big)^r+(1-\lambda)^{-pq}\,\Big(\calka{X}{|g|^p}\Big)^r,
\end{align*}
where $\lambda ^pY=\{\lambda ^py\colon y\in Y\}\subset Y.$ If $\calka{X}{|f|^p}=0$ or $\calka{X}{|g|^p}=0,$ 
we take the limit as $\lambda$ approaches  $0$ or $1$, respectively.  Otherwise, we obtain
$\eqref{c58}$ by minimizing the right-hand side with respect to  $\lambda$.    
\end{proof}

\begin{corollary}
Let $Y=[0,1]$, 
$Y=[0,\infty)$ or $Y=[0,\infty]$.
If $\mu$ is subadditive, then for all measurable functions $f,g\colon X\to \mR$ and $p>0$ we have 
\begin{align*}
\left(\sint_X |f+ g|^p\md\mu\right)^{1/(p+1)}&\le \left(\sint_X |f|^p\md\mu\right)^{1/(p+1)}+\left(\sint_X |g|^p\md\mu\right)^{1/(p+1)},\\
\left(\nint_X |f+ g|^p\md\mu \right)^{1/(p+1)}&\le \left(\nint_X |f|^p\md\mu\right)^{1/(p+1)}+\left(\nint_X |g|^p\md\mu\right)^{1/(p+1)},
\end{align*}
where $|f+g|^p,$ $|f|^p,$ $|g|^p\in Y$ and the integrals are defined, respectively,  by $\eqref{sug}$ and $\eqref{shi}.$
\end{corollary}

The next result deals with a~modified Shilkret integral and follows from Theorem $\ref{tw_subad}$ and the inequality $(x+y)^s\le x^s+y^s$ for $x,y\ge 0$ and $0<s<1.$  

\begin{corollary}
Let  $a\c _q b=(ab)^q$ with $0<q<1$ and let $Y=[0,1]$ or $Y=[0,\infty)$. For any subadditive measure $\mu$ and 
any measurable functions $f,g\colon X\to \mR$, we get
\begin{align*}
\Big(\calka[\c_q]{X}{|f+g|^p}\Big)^{1/p}&\le \Big(\calka[\c_q]{X}{|f|^p}\Big)^{1/p}+\Big(\calka[\c_q]{X}{|g|^p}\Big)^{1/p},
\end{align*} 
where $p=1/(1-q),$  and $|f+g|^p,$ $|f|^p,$ $|g|^p\in Y.$ 
\end{corollary}
Simple calculations show that 
$$
\Big(\calka[\c_q]{X}{|f|^p}\Big)^{1/p}=\sup _{t\in Y}\left\{ t^q\mu\big(\lbrace |f|\ge t\rbrace\big)^{q/p}\right\},
$$
so this functional is similar to a quasi-norm in the Lorentz type capacity spaces $\cite{cerda}.$

Now, we analyze the subadditivity of the Shilkret integral. Recall that a~monotone measure 
$\mu$ is {\it maxitive} if  for all disjoint sets $A,B\in \cA$ we have 
\begin{align}\label{max}
\mu (A\cup B)=\mu (A)\vee \mu (B).
\end{align}
Observe that $\mu$ is maxitive if and only if $\eqref{max}$ holds for all $A,B\in \cA$. 
In fact, if $\mu$ is maxitive and $A\cap B\neq \emptyset$, then 
$\mu (A\cup B)=\mu (A)\vee \mu (C)$ and $\mu (A\cup B)=\mu (D)\vee \mu (B),$
where $C=B\backslash A$ and $D=A\backslash B.$
This implies that 
$
\mu (A\cup B)=\mu (A)\vee \mu(B)\vee \mu (C)\vee \mu(D)=\mu (A)\vee \mu(B),
$
so  $\eqref{max}$ is satisfied. Clearly, any maxitive measure is subadditive. 

The following result can be found in $\cite{cat}$ (see also $\cite{shilkret}$ pp. $112$-$113$ and $\cite{cerda}$ Theorem $4.2$).

\begin{tw}\label{subShi} Let $Y=[0,1],$  $Y=[0,\infty)$ or $Y=[0,\infty].$ The Shilkret inegral $\eqref{shi}$ is subadditive 
for all measurable functions $f,g\colon X\to Y$  
if and only if the monotone measure $\mu$ is maxitive.  
\end{tw}
\begin{proof}
First, observe that $\mu$ is maxitive if and only if $\eqref{max}$ holds for all $A,B\in \cA$. 
In fact, if $\mu$ is maxitive and $A\cap B\neq \emptyset$, then 
$\mu (A\cup B)=\mu (A)\vee \mu (C)$ and $\mu (A\cup B)=\mu (D)\vee \mu (B),$
where $C=B\backslash A$ and $D=A\backslash B.$
This implies that 
$
\mu (A\cup B)=\mu (A)\vee \mu(B)\vee \mu (C)\vee \mu(D)=\mu (A)\vee \mu(B),
$
so  $\eqref{max}$ is satisfied. 
Denote the Shilkret integal for short by $\bI (f)$.

,,$\Leftarrow$'' We follow the proof of $\cite{cat,shilkret}.$ 
If $\bI(f)=\bI(g)=0$, then $\bI(f+g)=0$ as $\mu \big(\lbrace f+g\ge t\rbrace \big)\le \mu \big(\lbrace f\ge t/2\rbrace\big)+\mu \big(\lbrace g\ge t/2\rbrace\big)=0$ for all $t>0.$
Therefore, we  assume that 
$0<\bI(f)+\bI(g)<\infty, $ without loss of generality. By maxitivity of $\mu$,  we have 
\begin{align*}
t\mu (\{f+g\ge t\})&\le t\mu\big(\lbrace f\ge \lambda t\rbrace \cup \lbrace g\ge (1-\lambda )t\rbrace\big)\\
&=t\mu \big(\lbrace f\ge \lambda t\rbrace\big)\,\vee\,t\mu \big(\lbrace g\ge (1-\lambda )t\rbrace\big)
\end{align*}   
with $\lambda=\bI (f)/(\bI (f)+\bI(g)).$ Hence,
\begin{align*}\bI (f+g)\le \bigl((\bI (f)/\lambda\bigr)\vee \bigl(\bI (g)/(1-\lambda)\bigr)=\bI(f)+\bI(g).
\end{align*}   

,,$\Rightarrow$'' Suppose $\mu$ is not maxitive, i.e. $\mu (A\cup B)>\mu(A)\vee \mu (B)$ for some disjoint sets $A,B\in \cA.$ 
Thus, there exists $\lambda \in (0,1)$ such that  $\lambda \mu (A\cup B)>\mu(A)\vee \mu (B).$
Putting $f=\mI{A}+\lambda \mI{B}$, $g=(1-\lambda)\mI{B}$, we get
\begin{align*}
\bI(f)+\bI(g)&=\big((\lambda\mu(A\cup B))\vee \mu(A)\big)+(1-\lambda)\mu (B)\\
&<\lambda\mu(A\cup B)+(1-\lambda)\mu (A\cup B)=\bI(f+g),
\end{align*}  
so the Shilkret integral is not subadditive.
\end{proof}

Subadditivity of the Sugeno integral will be examined in the next section.



\section{Results for generalized lower Sugeno integral}
The {\it generalized lower Sugeno integral} 
of a~measurable function $f\colon X\to Y$ on a~set $D\in \cA$ with respect to a monotone measure $\mu$ and nondecreasing operator
$\c\colon Y\times \mu (\cA)\to [0,\infty]$ is defined as
\begin{align}\label{dol9}
\dcalka[\c]{D}{f}:=\inf_{t\in Y}\left\{ t\c\mu\big(D\cap \lbrace f> t\rbrace \big)\right\}.
\end{align}  
Observe that the functional $\eqref{dol9}$ is the universal integral in the sense of  Definition $2.5$ in $\cite{klement3}$ if 
$a\circ 0=a$ and $0\circ b=b$ for all $a\in Y $ and $b\in \mu(\cA).$
Putting $\c=\vee$  in $\eqref{dol9}$ we obtain the {\it lower Sugeno integral} $\cite{murofushi2}$
\begin{align}\label{dol911}
\dolsug{D}{f}:=\inf_{t\in Y} 
\left\{ t\vee \mu\big(D\cap \lbrace f> t\rbrace \big)\right\}.
\end{align}  
Mimicking the proof of  Theorem $5$ in $\cite{kandel}$ and Theorem $9.1$ in $\cite{zwang1}$ 
one can show that for any $Y=[0,m]\subset [0,\infty]$ the integral $\eqref{dol911}$ 
is equal to the Sugeno integral $\eqref{sug}$
\begin{align}\label{cd16}
\dolsug{D}{f}=\sint_{D} f\md\mu.
\end{align}
\medskip

{\bf Open problem 2.} Does there exist a~pair of operators $(\trd,\trr)\neq (\vee,\wedge)$ such that for all $f\colon X\to Y$  
\begin{align}
\dcalka[\trd]{D}{f}=\calka[\trr]{D}{f}\;?
\end{align}

\medskip
We say that measurable functions $f,g\colon X\to Y$ are 
$\mu$-{\it subadditive} for an~operator
$\trdd\colon \mu (\cA)^2\to \mu (\cA)$ and  a~set	
$D\in \cA$ if for all  
$a,b\in Y$
\begin{align*}
\mu\Big(D\cap \big(\{f> a\}\cup \{g> b\}\big)\Big)\le \mu\big(D\cap\left\{f> a\right\}\big)\trd \mu\big(D\cap\left\{g> b\right\}\big).
\end{align*}
Observe that $\mu$-subadditivity implies that $x\vee y\le x\trd y$ for all $x,y\in \mu (\cA\cap D)$. 

Now, we present several examples of $\mu$-subadditive functions.

\begin{example} \label{dol_ex4}\rm    
Any comonotone functions $f,g$ are $\mu$-subadditive with respect to an~operator $\trd$ such that $x\vee y\le x\trd y$ for all $x,y\in Y.$  For instance, any $t$-semiconorm $\rS^*$ on $Y=[0,1]$ has this property  (see $\cite{klement2}$).
\end{example}

\begin{example}\label{dol_ex4a} \rm Recall that $\mu$ is {\it submodular} if $\mu (A\cup B)\le \mu (A)+\mu (B)-\mu (A\cap B)$ for all $A,B\in \cA.$   Let $D=X$, 
$x\trd y=1-(1-x)(1-y)$ for $x,y\in Y=[0,1]$ and let
$\mu$ be a~submodular and monotone measure. Functions $f,g$ are 
$\mu$-subadditive if $f,g$ are positive quadrant dependent $\cite{boczek1}$, that is, 
$\mu\big(\{f>t\}\cap\{g>s\}\big)\ge \mu\big(\{f>t\}\big)\mu\big(\{g>s\}\big)$ for all $t,s\in Y$.
\end{example}

\begin{example}\label{dol_ex4b} \rm Put
$x\trd y=x+y$ for
$x,y\in Y=[0,\infty].$ Then  
any functions $f,g$ are
$\mu$-subadditive for a~subadditive measure $\mu$ on $X.$
\end{example}

Suppose   $\st,\loo\colon Y^2\to Y,$ and $\c_i\colon Y\times \mu (\cA)\to Y,$ $i=1,2,3,$ 
are nondecreasing and 
$\loo$ is right-continuous. Suppose also that  $\phi_i\colon Y\to Y$ is an increasing function and
$\phi_i(Y)=Y$ for  $i=1,2,3.$  

\begin{tw}\label{cdtw1} Assume that
for $a,b\in Y$ and $c,d\le \mu(D)$, we have
	\begin{align}\label{cd6}
\phi^{-1}_1\big(\phi_1(a\st b)
\c_1 (c \trd d)\big)\le \phi_2^{-1}\big(\phi_2(a)\c_2 c\big)\,\lo\, \phi_3^{-1}\big(\phi_3(b)\c_3 d\big).
	\end{align} 
If $f,g$ are $\mu$-subadditive for $\trd$
and~$D$, then 
\begin{align}\label{cd7}
\phi_1^{-1}\bigg(\dcalka[\caa]{D}{\phi _1\big(f\st g\big)}\bigg)\le \phi_2^{-1}\bigg(\dcalka[\cbb]{D}{\phi_2(f)}\bigg)\,
\lo\, \phi_3^{-1}\bigg(\dcalka[\ccc]{D}{\phi_3(g)}\bigg).
\end{align}
\end{tw}

\begin{proof} 
By the monotonicity of $\star$ and $\mu$,  
for any $D\in \cA$ we obtain
\begin{align*}
\mu\big(D\cap \left\{f\st g>a\st b\right\}\big)\le\mu\Big( D\cap \big(\{f>a\}\cup \{g>b\}\big)\Big).
\end{align*}
From~$\mu$-subadditivity of  $f,g$ and from the fact that $b\mapsto a\c_1b$ is a nondecreasing function we get
\begin{align}\label{cd10}
\phi _1&(a\st b)\c_1\mu\big( D\cap\left\{\phi _1(f\st g)>\phi _1(a\st b)\right\}\big)\nonumber\\&\le 
\phi _1(a\st b)\circ_1\Big(\mu\big(D\cap\lbrace \phi _2(f)>\phi_{2}(a)\rbrace\big)\trd \mu\big(D\cap \lbrace \phi _3(g)>\phi_{3}(b)\rbrace \big)\Big).
\end{align}
By $\eqref{cd6}$ and $\eqref{cd10}$
\begin{align}\label{cd11}
\phi^{-1}_1\Big(\phi _1&(a\st b)\circ _1 \mu\big(D\cap\lbrace \phi _1(f\st g)>\phi _1(a\st b)\rbrace\big)\Big)\nonumber
\\\le &\phi_{2}^{-1}\Big(\phi_{2}(a)\circ _{2}\mu\big(D\cap\lbrace \phi_{2}(f)>\phi_{2}(a)\rbrace\big)\Big)\lo
\phi_{3}^{-1}\Big(\phi_{3}(b)\circ _{3}\mu\big(D\cap\left\{\phi_{3}(g)>\phi_{3}(b)\right\}\big)\Big).
\end{align}
Since $\phi_1^{-1}$ is increasing, we have
\begin{align*}
\phi_1^{-1}&\bigg(\dcalka[\caa]{D}{\phi _1 (f\st g )}\bigg)\\
&\le \phi_{2}^{-1}\Big(\phi_{2}(a)\circ _{2}\mu\big(D\cap\lbrace\phi_{2}(f)>\phi_{2}(a)\rbrace\big)\Big)\lo
\phi_{3}^{-1}\Big(\phi_{3}(b)\circ _{3}\mu\big(D\cap\lbrace \phi_{3}(g)>\phi_{3}(b)\rbrace\big)\Big)
\end{align*}
for all $a,b \in Y$.  Taking the infimum over $a\in Y,$ we get
\begin{align*}
\phi_1^{-1}\bigg(\dcalka[\caa]{D}{\phi _1\big(f\st g\big)}\bigg)\le 
\phi_2^{-1}\bigg(\dcalka[\cbb]{D}{\phi _2(f)}\bigg)
\lo\, \phi_3^{-1}\Big(\phi_3(b)\circ _3\mu\big(D\cap \left\{\phi_3(g)>\phi_3(b)\right\}\big)\Big).
\end{align*}
Proceeding similary with the infimum in $b\in Y$, we obtain $\eqref{cd7}.$
\end{proof}

\begin{example}\label{cdex1}
We know from Example $\ref{dol_ex4}$ that any comonotone 
 functions $f,g\colon X\to Y$  are $\mu$-subadditive with $\trd =\vee$. Put
$Y=[0,\infty].$
If 
\begin{align}\label{cd14}
(a\st b)\vee\big(\phi_1^{-1}(c)\vee\phi_1^{-1}(d)\big) \le \big(a\vee \phi_2^{-1}(c)\big)\st\big(b\vee \phi_3^{-1}(d)\big),
\end{align}
then for all $D\in\cA$ we get 
\begin{align}\label{cd15}
\phi_1^{-1}\bigg(\dolsug{D}{\phi_1(f\st g)}\bigg)\le \phi_2^{-1}\bigg(\dolsug{D}{\phi_2(f)} \bigg)\st \phi_3^{-1}\bigg(\dolsug{D}{\phi_3(g)}\bigg).
\end{align}
The inequality $\eqref{cd14}$ is satisfied for any operator $\st$ such that  $a\st b\ge a\vee b$ and functions
$\phi_1\ge\phi_i,$ $i=2,3.$ Indeed, combining 
$a_1\st a_2\le \big(a_1\vee \phi_2^{-1}(b_1)\big)\st \big(a_2\vee \phi _3^{-1}(b_2)\big)$ with
$$\phi_1^{-1}(b_1)\vee\phi_1^{-1}(b_2)\le \phi_1^{-1}(b_1)\st\phi_1^{-1}(b_2)\le \big(a_1\vee \phi_2^{-1}(b_1)\big)\st \big(a_2\vee \phi_3^{-1}(b_2)\big)$$
yields $\eqref{cd14}.$ 
From $\eqref{cd16}$  and $\eqref{cd15}$ we can get a~generalization of  Theorem $3.1$ in $\cite{lwu}.$ 
\end{example}

\begin{example}\rm 
Let $Y=[0,1],$ $D=X$ and $\mu(X)=1.$ Put $x\trd y=x+y-xy$ and $x\star y=x\lo y=(x+y)\wedge 1,$  where $x,y\in Y.$ If $\circ _i=\vee$ and $\phi_i(x)=x,$ $i=1,2,3,$ then the condition $\eqref{cd6}$ takes the form
\begin{equation}\label{dol10}
\big((a+b)\wedge 1\big)\vee (c+d-cd)\le (a\vee c+b\vee d)\wedge 1,
\end{equation}
$a,b,c,d\in Y.$
Since $(a+b)\wedge 1\le a+b\le (a\vee c)+(b\vee d)$ and 
$c+d-cd\le (a\vee c+b\vee d),$
the inequality $\eqref{dol10}$ is true for all $a,b,c,d$.
Hence, if $f,g\colon X\to [0,1]$ are positive quadrant dependent  functions  
with respect to a submodular and monotone measure  $\mu$ on $X$ (see Example $\ref{dol_ex4a}$), then  
\begin{align}\label{dol11}
\dolsug{X}{(f+g)\wedge 1}\le \dolsug{X}{f}+\dolsug{X}{g}.
\end{align}
\end{example}

\begin{example}\label{dol_ex6}\rm 
Set $Y=[0,\infty]$ and $\trd =\star =\loo=+.$ Let $\mu$  be a~subadditive measure,  
$\circ _i=\vee $ and $\phi _i(x)=x$ for all $i$. Then 
\begin{align}\label{dol12}
\dolsug{X}{(f+g)} \le\dolsug{X}{f} +\dolsug{X}{g}
\end{align}
for all $f,g\colon X\to Y$.
\end{example}

Next, we prove that the subadditivity property of Sugeno integral $\eqref{sug}$ 
with $Y=[0,1]$ or $Y=[0,\infty)$ is  
equivalent to  subadditivity of a finite measure $\mu.$  

\begin{tw}\label{dol_twsub} If $\mu$ is subadditive, then  
\begin{align}\label{subS}
\sint_{X} (f+g)\md\mu\le \sint_X f\md\mu+\sint_X g\md\mu
\end{align}
for all $f,g\colon X\to Y.$ Moreover,  
if $\eqref{subS}$ holds for all measurable functions $f,g\colon X\to Y$ and $\mu (X)<\infty$, then $\mu$ is 
subadditive.   
\end{tw}
\begin{proof} The inequality $\eqref{subS}$ follows immediately from  
$\eqref{dol12}$ and $\eqref{cd16}$.
Moreover, let $f=a\mI{A},$ $g=a\mI{B},$ where $a\ge 0,$ $A,B\in\cA.$   
From $\eqref{subS}$ and the finite and monotone measure $\mu,$ we have \begin{align*}
\mu(A\cup B)=a\wedge \mu(A\cup B)\le \big(a\wedge \mu(A)\big)+\big(a\wedge\mu(B)\big)\le \mu(A)+\mu(B)
\end{align*}
for $a\ge \mu(A\cup B)$, which completes the proof.
\end{proof} 

The assumption $\mu(X)=\infty$ in Theorem $\ref{dol_twsub}$ cannot be omitted. Indeed, 
if $\mu$ is a~subadditive  measure such that $\mu (X)=\infty$ and $\mu(A),\mu (B)<\infty$ for some $A,B,$  such that $A\cup B=X,$
 then  the inequality $\eqref{subS}$ is not true for 
$f=a\mI{A},$ $g=a\mI{B}$ and $a>\max (\mu (A),\mu (B))$.  


\medskip
Now, we show that from the Minkowski-H\"{o}lder type inequality for  integral $\eqref{calka1}$ one can obtain
an~inequality of the same type for  integral $\eqref{dol9}$ and vice versa. Suppose 
$Y=[0,m]$, $0<m\le \infty$. Let $h\colon Y\to Y$ be a decreasing function such that $h(Y)=Y,$  
$h(0)>0$ and~$h\big(m)=0$. For instance,   $h(x)=1-x$ for $Y=[0,1]$ and 
$h(x)=1/x$ for  $Y=[0,\infty]$ under convention that  $1/0=\infty$ and $1/\infty=0.$ Suppose 
$\mu_h$ is a~monotone measure on 
$(X,\cA)$ defined as
$\mu_h(A)=h^{-1}\big(\mu(X\backslash A)\big).$
For a given operator  
$\circ\colon Y^2\to Y$ 
let us define the operator  
 \begin{align*}
a \ch b=h^{-1}\big(h(a)\circ h(b)\big),\quad a,b\in Y.
\end{align*}
For any measurable function
$f\colon X\to Y,$ we have 
\begin{align}\label{dol13}
h^{-1}\bigg(\dcalka[\c]{X}{h(f)} \bigg)&=
h^{-1}\Big(\inf _{y\in Y}\left\{h(y)\c \mu(\{h(f)>h(y)\}\big)\right\}\Big)\nonumber\\&=
\inf _{y\in Y}\left\{h^{-1}\Big(h(y)\c \mu(\{f\le y\}\big)\Big)\right\}=\calka[\c_h]{X}{f}.
\end{align}
Applying the formula $\eqref{dol13}$ and Theorem  $\ref{ctw7}$ with  $\star=\loo$ and $\phi _i(x)=x$ for all $i$ gives the following Corollary.  

\begin{corollary}\label{dol_colh}  Assume $\c\colon Y^2\to Y$ is nondecreasing, $m\c y=y\c m=m$ for all $y\in Y=[0,m]$ 
and $f,g\colon X\to Y$ are $\star$-associated. The following inequality is satisfied 
\begin{align*}
h^{-1}\bigg(\dcalka[\c]{X}{h(f\star g)}\bigg)\le 
h^{-1}\bigg(\dcalka[\c]{X}{h(f)}\bigg)\star 
h^{-1}\bigg(\dcalka[\c]{X}{h(g)}\bigg)
\end{align*}
if and only if  
$(a\star b)\c _h c\le (a\c _h c)\star (b\c _h c)$
for all $a,b,c\in Y.$
\end{corollary}  
 \medskip

\begin{example}\rm  From the well-known inequality 
$(x+y)/(1+x+y)\le \big(x/(1+x)\big)+\big(y/(1+y)\big)$ for $x,y\ge 0,$ it follows that for all $a,b,c\ge 0$
$$\big((a+b)^{-1}+c^{-1}\big)^{-1}\le \big(a^{-1}+c^{-1}\big)^{-1}+\big(b^{-1}+c^{-1}\big)^{-1}$$
with $1/0=\infty$ and $1/\infty=0.$ This implies that  the necessary and sufficient condition of Corollary $\ref{dol_colh}$ is satisfied for 
 $Y=[0,\infty]$, $h(x)=x^{-1}$ and $\star=\circ =+$. Thus,  
for any comonotone functions $f,g\colon X\to Y,$ we have 
\begin{align*}
\bigg(\dcalka[+]{X}{1/(f+g) }\bigg)^{-1}\le 
\bigg(\dcalka[+]{X}{1/f}\bigg)^{-1}+ 
\bigg(\dcalka[+]{X}{1/g}\bigg)^{-1}.
\end{align*}
\end{example}

The next result is an~immediate consequence  of Theorem $\ref{cdtw1}$ with $\phi _i(x)=x$ for all $i$ and the formula 
\begin{align*}
\dashint_{\c_h,X} f\md \mu_h=h^{-1}\bigg(\calka{X}{h(f)} \bigg).
\end{align*}

\begin{corollary}\label{dol_colh2}  Assume that $f,g\colon X\to Y$ are $\mu _h$-subadditive for 
$\trd ,$  operator $\star $ is nondecreasing and right-continuous and  $\c$ is nondecreasing. Assume also that  
$(a\star b)\c_h (c\trd d)\le (a\c_h c)\star (b\c_h d)$ for all $a,b,c,d\in Y$. Then  
\begin{align*}
h^{-1}\Big(\calka{X}{h(f\st g)} \Big)\le 
h^{-1}\Big(\calka{X}{h(f)}\Big)\star 
h^{-1}\Big(\calka{X}{h(g)}\Big).
\end{align*}
\end{corollary}
 \medskip

\begin{example} Suppose $\mu (A)=1/\mu_h(X\backslash A)$ for $A\in \cA.$ 
From Example $\ref{dol_ex6},$ formula $\eqref{cd16}$ and Corollary $\ref{dol_colh2}$ for $h(x)=1/x$, $\c _h=\vee$ and $\star=+$, it follows that for any measurable functions  
$f,g\colon X\to [0,\infty]$, we have  
\begin{align*}
\bigg(\sint _X  1/(f+g) \md \mu\bigg)^{-1}\le 
\bigg(\sint _X 1/f\md \mu\bigg)^{-1}+ 
\bigg(\sint _X 1/g\md \mu\bigg)^{-1},
\end{align*}
with $\sint _X $ being the Sugeno integral $\eqref{sug}$ for $Y=[0,\infty]$ and $1/\infty=0$, $1/0=\infty.$
\end{example}

\section{Application}
As an application of the results of  this paper, we provide new  metrics  in the space of $\cA$-measurable functions $f\colon X\to \mR$ defined  on a~fuzzy space $(X,\cA,\mu).$ First, let us recall some  facts. Taking $\c=+$ and $Y=[0,\infty]$ in $\eqref{dol9},$ we get the functional  
\begin{align*}
d_F(\sX,\sY)=\inf _{\varepsilon \ge 0}\left\{\varepsilon+\mu\big(\lbrace |\sX-\sY|>\varepsilon\rbrace \big)\right\}
\end{align*}
on the space $L^0(X)$ of all random variables defined on a~probability space $(X,\cA,\mu).$  This functional was
proposed by Fr\'{e}chet  $\cite{fre}$ in order to metrize the convergence in 
measure $\mu$ (see also $\cite{ca}$, p. $356$, and $\cite{sch}$, pp. $101-104$).
The integral $\eqref{dol9}$ with $\c=\vee$ was  introduced by Ky Fan $\cite{fan}$. 
He proved that $L^0(X)$ with the metric 
\begin{align*}
d_{KF}(\sX,\sY)=\inf \left\{\varepsilon \ge 0\colon  \mu\big(\lbrace |\sX-\sY|>\varepsilon\rbrace\big)\le \varepsilon\right\}
\end{align*}
is a~complete space.  By $\eqref{cd16}$ we have
\begin{align*}
d_{KF}(\sX,\sY)=\dolsug{X}{|\sX-\sY|} =\sint _{X}|\sX-\sY|\md\mu. 
\end{align*}
Li $\cite{gli}$ extended  Ky Fan's result to cover the case of any continuous from below, finite and subadditive measure $\mu$.
 
Now we are ready to introduced new metrics.  
Given $p>0,$ let  $Y=[0,\infty]$ and   
$\c\colon Y^2\to Y$ be a~non-decreasing operator such that  
$x\c (y+z)\le (x\c y)+(x\c z)$ and  $(ax)\c y\le a^p(x\c y)$ for $x,y,z\in Y$ and $a>1.$  
We also assume that if $1\c x\le y$ for $0<y<1,$ then $x\le y$.
For instance, $x\c y=x^p\wedge y^u$   or $x\c y=x^py^u,$ where $0<u\le 1.$ Suppose  $\mu$ is a~subadditive measure and put 
\begin{align}\label{met2}
D_{\c ,p}(f,g)=\bigg(\calka{X}{|f-g|^p}\bigg)^{1/(p^2+1)}.
\end{align}
As special cases we get 
 \begin{align}\label{calki1}
D_{\wedge ,1}(f,g)=\bigg(\sint_X |f-g|\md\mu \bigg)^{1/2},\quad
D_{\cdot,1}(f,g)=\bigg(\nint_X |f-g|\md\mu \bigg)^{1/2}.
\end{align}
Denote by 
$\mathcal{L}_{\c}^p$ the class of measurable functions $f\colon X\to \mR$ such that 
$D_{\c ,p}(f,0)<\infty$.  
Let $f\sim g$ mean that $\mu (\{|f-g|>0\})=0$ and let $L^p_{\circ}$ be the set of the equivalence classes in $\mathcal{L}_{\circ}^p$ determined by 
the equivalence relation $\sim$. If $[f]$ is the equivalence class containing $f$, define $d_{\c ,p}(f,g)=D_{\c ,p}([f],[g]).$  
 
\begin{tw}\label{doltw2}Suppose $\circ\colon Y^2\to Y$ is left-continuous in the second argument.  If  $\mu$ is subadditive and continuous from below, then 
$\big(L^p_{\circ},d_{\c,p}\big)$ is a~complete metric space. 
 \end{tw} 

To prove Theorem $\ref{doltw2}$ we need the monotone convergence theorem and Fatou's lemma for 
the integral $\eqref{calka1}.$ We recall that  $\mu$ is {\it null-additive} if $\mu (A) = 0$ implies $\mu (A\cup B) = \mu(B)$ for every $B\in \cA$. Observe that if $\mu$ is subadditive, then it is also null-additive.

\begin{lemma}[Monotone convergence]\label{beppo}
Let  $\circ\colon Y^2\to Y$ be left-continuous in the second argument. If $\mu$ is a continuous from below, null-additive and monotone measure and if $(f_n)$ is a~sequence of functions $f_n\colon X\to Y$ which is nondecreasing and converges to $f$ on $A^c=X\backslash A$ with $\mu(A)=0,$  then
$\lim_{n\to\infty}\calka{X}{f_n}=\calka{X}{f}.$
\end{lemma}
\begin{proof} Measure $\mu$ is null-additive, so $\calka{X}{g}=\calka{A^c}{g} $ for any $g.$
The rest of the proof is similar to that of Lemma 14 in   $\cite{chateauneuf2}.$
\end{proof}

\begin{lemma}[Fatou]\label{fatou}
Suppose  $\circ\colon Y^2\to Y$ is left-continuous in the second argument,
$f_n\colon X\to Y$ for all $n.$ If
$\mu$ is a continuous from below, null-additive and monotone measure and  $f(x)=\lim_{n\to\infty} f_n(x)$ for all $x\in A^c$ with $\mu (A)=0,$ then
$\calka{X}{f} \le \liminf_{n\to\infty}\calka{X}{f_n}.$
\end{lemma}

\begin{proof} The proof follows from Lemma $\ref{beppo}$ and standard arguments (see \cite{kallenberg}, Lemma 1.20). 
\end{proof}

\begin{proof}[Proof of Theorem $\ref{doltw2}$] Assume $c\c d>0$ for some $c,d>0$; otherwise the result is trivial. We will show that $x\c y>0$ for all $x,y>0$. In fact, suppose 
that $x\c y=0$ for some $x,y>0.$  Then
$x\c t\le x\c y=0$ for all $0\le t\le y,$ so from the subadditivity of $t\mapsto x\c t$ it easily follows that $x\c t=0$ for all $t\in Y.$ Next, 
by the monotonicity of $\c$, we have $s\c t\le x\c t=0$ for all $0\le s\le x$ and $(ax)\c t\le a^p(x\c t)=0$ for any $a>1$. Hence so $s\c t=0$ for all $s,t\in Y,$ a~contradiction.

Next, suppose  $d_{\c ,p}(f,g)=0.$ Hence, $\mu \big(\lbrace |f-g|\ge t\rbrace\big)=0$ for all $t>0.$ 
Since $\mu $ is continuous from below, we have $\mu (\{|f-g|>0\})=0,$ so $f\sim g$.
Clearly $d_{\c,p}$ is symmetric and  it follows from Theorem $\ref{tw_subad}$ for $r=1$ and $q=p$ that $d_{\c ,p}$ satisfies the triangle inequality. The proof of the completeness is a~modified version of that of
Lemma $1.31$ in $\cite{kallenberg}.$
Given a~Cauchy sequence $(f_n),$ let $(f_{n(k)})$ 
be a~subsequence such that 
$\big(d_{\c ,p}(f_{n(k+1)},f_{n(k)})\big)^{p^2+1}\le 4^{-kp}.$ Put $A_k=\{x\in X\colon |f_{n(k+1)}(x)-f_{n(k)}(x)|^p\ge 2^{-k}\}.$ Then 
$2^{-k}\c \mu (A_k)\le 4^{-kp}$ by the definition of $d_{\c,p}.$ From the property $(2x)\c y\le 2^p(x\c y)$ we get
$$
1\c \mu (A_k)\le (2^{p})^k\bigl(2^{-k}\c \mu (A_k)\bigr)\le 2^{-kp}. 
$$ 
By the assumption on $\c$, $\mu (A_k)\le 2 ^{-kp}$ for all $k$.
Set $A=\bigcap _{i=1}^\infty \bigcup _{k=i}^\infty A_k$. Since $\mu$ is subadditive, we have 
$$
\mu (A)\le \mu \bigg(\bigcup _{k=i}^\infty A_k\bigg)\le \sum _{k=i}^\infty \mu (A_k)\to 0\;\; \hbox{as}\; i\to \infty, 
$$  
so $\mu (A)=0.$ Let $x\in A^c=X\backslash A.$ 
Since $|f_{n(k+1)}(x)-f_{n(k)}(x)|<2^{-k/p}$ for 
all large enough $k,$ we have 
$$
\sup _{r\ge k}|f_{n(r)}(x)-f_{n(k)}(x)|\le \sum _{r=k}^\infty 
|f_{n({r+1})}(x)-f_{n(r)}(x)|\le \frac{2^{-k/p}}{1-2^{-1/p}}.
$$
Thus, $\big((f_{n(k)}(x)\big)_{k=1}^\infty$ is a~Cauchy sequence and  
$f_{n(k)}(x)\to f(x)$ for all $x\in A^c,$ where $f$ is  some measurable function (if $A\neq \emptyset$, define $f$ on $A$ arbitrarily). 
We recall that any subadditive measure  $\mu$ is also null-additive. 
By  Lemma $\ref{fatou}$, we get
\begin{align*}
d_{\c ,p}(f,f_n)\le \liminf _{k\to \infty}d_{\c,p}(f_{n(k)},f_n)\le 
\sup _{m\ge n}d_{\c,p}(f_m,f_n)\to 0,\;n\to \infty,
\end{align*}
as $(f_n)$ is a~Cauchy sequence. This shows that $f_n\to f$ in metric $d_{\c,p}.$
\end{proof}

Put $\|f \|=\nint_X |f|\md\mu$ and denote by $L^1_N$ the space of all 
functions $f\colon X\to \mR$ such that $\|f\|<\infty.$

\begin{corollary} If  $\mu$ is maxitive, then  
$L^1_N$ is a~Banach space with the norm $\|\cdot\|$.
\end{corollary}

\begin{proof} Any maxitive measure is subadditive. Observe that 
$\|f-g\|=D_{\cdot,1}(f,g)^2$ (see $\eqref{calki1}$) and 
$\|cf\|=|c|\|f\|$ for $c\in \mR$. The result follows immediately 
from Theorems $\ref{subShi}$ and $\ref{doltw2}$. 
\end{proof}  
 
The next theorem is a~counterpart of  Theorem $2.21$ in $\cite{hutnik2}$

\begin{tw}\label{doltw3} Adopt the assumptions of Theorem $\ref{tw_subad}$ with $Y=[0,\infty]$ and some $p>0$.  If $f_n,f\in L^p_\c$ for  all $n$ and 
$\lim _{n\to \infty}d_{\c,p}(f_n,f)=0$, then 
$\calka{X}{f_n^p} \to \calka{X}{f^p}$
as $n\to \infty$. 
\end{tw}

\begin{proof} From Theorem $\ref{tw_subad}$ we get
$d_{\c,p}(f_n,0)\le d_{\c,p}(f_n,f)+d_{\c,p}(f,0)$ and 
$d_{\c,p}(f,0)\le d_{\c,p}(f_n,f)+d_{\c,p}(f_n,0).$
This implies
\begin{align*}
\left|\Big(\calka{X}{f_n^p}\Big)^{1/(p^2+1)}- \Big(\calka{X}{f^p}\Big)^{1/(p^2+1)} \right|\le d_{\c,p}(f_n,f),
\end{align*}
which completes the proof.
\end{proof}

Theorem  $\ref{doltw3}$ gives a~partial answer to the open problem $2.22$ in $\cite{hutnik2}$ as there exist discontinuous and subadditive measures. 


\end{document}